\documentclass[12pt,a4paper]{article}
\usepackage[utf8]{inputenc}
\usepackage[english]{babel}
\usepackage[T1]{fontenc}
\usepackage{amsmath}
\usepackage{amsthm}
\usepackage{amsfonts}
\usepackage{amssymb}
\usepackage{dsfont}
\usepackage{graphicx}
\usepackage{hyperref}
\usepackage{enumerate}
\usepackage{enumitem}
\usepackage{subfigure}
\usepackage{xcolor}
\allowdisplaybreaks

\usepackage{combelow}% provides \cb to place comma below character
\usepackage{placeins}

\usepackage{newunicodechar}
\newunicodechar{ș}{\cb{s}}
\newunicodechar{ț}{\cb{t}}

\newtheorem{theorem}{Theorem}[section]

\newtheorem{lemma}{Lemma}[section]
\newtheorem{proposition}{Proposition}[section]

\newtheorem{remark}{Remark}[section]

\def\R {{\mathbb R}}

\def\N {{\mathbb N}}
\def\E {{\mathbb E}}

\def\cB {{\cal B}}

\def\cU {{\cal U}}

\def\cF {{\cal F}}

\DeclareMathOperator*{\argmin}{arg\,min}

\normalbaselines

\begin{document}

\title{\bf A Nonlocal Stochastic Optimal Control with Elliptic-Type Smoothing}

\author{
\cb{S}tefana-Lucia Ani\cb{t}a{\thanks{
``Octav Mayer'' Institute of Mathematics of the Romanian Academy, Ia\cb{s}i 700505, 
Romania;}\ \thanks{Apollonia University,  Ia\c{s}i 700511, Romania; stefi\_anita@yahoo.com}}
\and
Luca Di Persio{\thanks{Department of Computer Science, University of Verona, Strada le Grazie 15, 37134 Verona, Italy;
luca.dipersio@univr.it}}
}

\date{}

\maketitle

\begin{abstract}
We study a nonlocal stochastic optimal control problem in which the control acts through the elliptic
smoothing operator $\mathcal{S}=(I-\Delta)^{-1}$ on $\mathbb{R}^d$. The state is described by the law of a controlled diffusion, 
equivalently by a controlled Fokker-Planck equation with coefficients depending on $\mathcal{S}u$. We prove the existence of an optimal control 
by the direct method in the calculus of variations. We then derive a Pontryagin-type maximum principle by spike variations, relying on sharp
properties of the operator $\mathcal{S}$. The result yields a pointwise minimisation rule for the optimal control. We discuss the meaning of the
optimal control problem in the case of population dynamics.
\end{abstract}



 \bigskip
 
%%%%%%%%
 
\section{Introduction}
\label{sec:intro}

Controlling a stochastic system through its diffusion coefficient, and not merely through its drift, has become a central theme in
mean‐field control, mathematical finance, and population dynamics. The recent article \cite{AnitaDiPersio2025} addresses this scenario when the
control acts \emph{non‐locally} by convolution. In the present work, we replace the convolution kernel with the resolvent of
the Laplacian, thereby introducing the \emph{elliptic smoother} $\mathcal{S}$. Besides its intrinsic physical meaning - stationary
diffusion–decay of an injected substance - $\mathcal{S}$ enjoys strong compactness properties that substantially streamline the proof of the
maximum principle while preserving the modelling features of non‐local control.
\vspace{5mm}

Consider the following deterministic optimal control problem
$$\underset{u \in {\cal U}_{\beta }}{\text{Minimize}} \ J(u)\  ,\leqno({\bf P})$$
where $$J(u)=\int_0^T \int_{\mathbb{R}^d}G(t,x)\rho ^u(t,x) \, dx \, dt+\int_{\mathbb{R}^d}G_T(x)\rho ^u(T,x) \, dx
+\int_{\mathbb{R}^d}\langle x\rangle ^{2\beta }h(u(x)) \, dx,$$ 
$T\in (0,+\infty )$, $d\in \mathbb{N}^*$ and $\rho ^u$ is the weak solution to the following controlled Fokker-Planck equation:
\begin{equation}\label{FPeq}
\left\{ \begin{array}{ll}
\displaystyle {{\partial \rho}\over {\partial t}}(t,x)=-\nabla \cdot (b(x,({\cal S}u)(x))\rho (t,x))+\frac{1}{2}\frac{\partial ^2}{\partial x_i\partial x_j}(q_{ij}(x,({\cal S}u)(x))\rho (t,x)),  \vspace{2mm} \\
\hspace{7cm} x\in \mathbb{R}^d, \ t\in (0,T), \vspace{2mm} \\
\rho (0,x)=\rho _0(x), \qquad x\in \mathbb{R}^d .
                                     \end{array}
                          \right. 
\end{equation}
Model (\ref{FPeq}) describes the dynamics of a population of cells or insects that is subject to a long-range aggregation described by $-\nabla \cdot (b \, \rho)$ and to a short-range
dispersion described by $\frac{1}{2}\frac{\partial ^2}{\partial x_i\partial x_j}(q_{ij}\rho)$; $\rho (t,x)$ is the probability density of the population at moment $t$ and position $x$.

Here, $u$ is an input that belongs to the set of controllers
$${\cal U}_{\beta}=\{v\in L^{\infty}(\R^d)\cap  L^2_{\beta}(\R^d); \ 0\leq v \leq M_0 \quad \mbox{\rm a.e. } x\in \R^d\} ,$$
where $\beta , M_0>0$, $\langle x\rangle =(1+|x|^2)^{\frac{1}{2}}$,
$$L^2_{\beta}(\R^d)=\{f; \langle x \rangle ^{\beta}f\in L^2(\R^d)\}  \quad  (\int_{\R^d}(1+|x|^2)^{\beta}|f(x)|^2dx<+\infty).$$
We denote by $\|f\|_{2,\beta}=\Big(\int_{\R^d}(1+|x|^2)^{\beta}|f(x)|^2dx\Big) ^{\frac{1}{2}}$ the norm of $L^2_{\beta}(\R^d)$.

If $u(x)$ is the infusion per unit of time (the flow rate) of a substance, with a diffusion coefficient $1$ and a decay rate equal to $1$, then the density of the substance when
the equilibrium is attained is given by $\eta = \mathcal{S}u\in H^1(\R^d)$ satisfying $\eta-\Delta \eta =u$ in $\R^d$.
There exists a unique weak solution $\eta \in H^1(\R^d)$ (i.e. $\displaystyle \int_{\R^d}\left[ \eta \, \zeta+\nabla\eta\cdot\nabla\zeta \right] \, dx = \int_{\R^d} u \, \zeta \, dx, \forall \zeta \in H^1(\R^d)$) 
via Lax-Milgram's lemma.

The substance modifies both the drift and the diffusion via the scalar field ${\cal S}u$ produced by $u$.

Problem $(P)$ asks to optimally displace a population via an action produced by a constant-in time infusion of a substance. The fact that $u\in L^2_{\beta }(\R^d)$
shows that the magnitude of the control decays fast as $|x|\rightarrow \infty $ and the presence of the weight $\langle x\rangle ^{2\beta }$ in the third term of the cost functional
shows that the cost associated to the control increases fast as $|x| \rightarrow \infty $.

Assume in what follows that the following hypotheses hold.

\begin{itemize}
\item[\bf (H1)] $b\in C^{1,2}_b(\R^d\times [0,M_0];\R^d)$.
\item[\bf (H2)] $q\in C^{2}_b(\R^d\times [0,M_0];\R^{d\times d})$, $q(x,s)=q(x,s)^T$ for all $(x,s)$, and there exists $\gamma>0$ such that
$$\gamma |\xi|^2\leq \xi^Tq(x,s)\xi,\quad \forall (x,s)\in \R^d\times [0,M_0],\quad \forall \xi\in \R^d.$$
When the stochastic formulation is used, we assume in addition that there exists $\sigma\in C^2_b(\R^d\times[0,M_0];\R^{d\times N})$, where $N\in \N^*$, such that $q=\sigma\sigma^T$.

\item[\bf (H3)] $\rho _0\in L^2(\R^d)\cap L^1(\R^d)$, $\rho _0(x)\geq 0$ a.e. in $\R^d$, and $\int_{\R^d}\rho_0(x)\,dx=1$.
\end{itemize}

Consider now the following related stochastic optimal control problem
$$\underset{u \in {\cal U}_{\beta }}{\text{Minimize}} \  \mathbb{E}\Big[\int_0^T G(t,X^u(t)) \, dt+G_T(X^u(T)) \Big]
+\int_{\mathbb{R}^d}\langle x\rangle ^{2\beta }h(u(x)) \, dx,\leqno({\bf P_S})$$
where $X^u$ is the solution to the following controlled stochastic differential equation
\begin{equation}\label{SDE}
\left\{ \begin{array}{ll}
\displaystyle d \, X(t)=b(X(t),({\cal S}u)(X(t))) \, dt+\sigma (X(t),({\cal S}u)(X(t))) \, dW(t) , \quad t\in (0,T) \, , \\ 
X(0)=X_0,
                                     \end{array}
                          \right. 
\end{equation}
where $(\Omega , {\cal G},\mathbb{P})$ is a probability space endowed with a normal filtration $({\cal G}_t)_{t\in [0,T]}$ satisfying the usual conditions and $(W(t))_{t\in [0,T]}$ 
is an $\R^N-({\cal G}_t)_{t\in [0,T]}$ Brownian motion and $X_0$ is an
$\R^d$-random variable of density $\rho _0$ satisfying $\E [|X_0|^2]<\infty $. One assumes that
$\sigma \in C^{2}_b(\R^d\times [0,M_0];\R^{d\times N})$ and that $q=\sigma \sigma ^T$.

Actually, (\ref{SDE}) describes the trajectory of an individual of the above mentioned population.

The properties of $\mathcal{S}$ will show that for any $u\in {\cal U}_{\beta }$, (\ref{SDE}) has a unique solution $X^u(t), \, t\in [0,T]$ (see \cite{Evans}, \cite{oksendal}) and that $X^u(t)$ has a probability density $\rho ^u(t,x)$
which is the unique weak solution to (\ref{FPeq}) (see \cite{RoecknerZhang2010}). Moreover, we assume that $G$, $G_T$ and $h$ satisfy

\begin{itemize}
\item[\bf (H4)] $G\in L^2((0,T)\times \mathbb{R}^d)$ and $G_T\in L^2(\mathbb{R}^d)$, with
$$G(t,x)\geq 0 \ \mbox{\rm a.e. in } (0,T)\times \mathbb{R}^d, \quad G_T(x)\geq 0 \ \mbox{\rm a.e. in } \mathbb{R}^d.$$
Moreover, $h:[0,M_0]\longrightarrow [0,+\infty)$ is convex and continuous and $h(0)=0$.
\end{itemize}
The condition $h(0)=0$ describes that the zero control has a zero associated cost and guarantees a finite cost. Thus, the infimum in problem $(P)$ is finite. 

So,
$$\mathbb{E}\Big[\int_0^T G(t,X^u(t)) \, dt+G_T(X^u(T)) \Big]+\int_{\R^d}\langle x\rangle^{2\beta}h(u(x)) \, dx$$
$$=\int_0^T \int_{\mathbb{R}^d}G(t,x)\rho ^u(t,x) \, dx \, dt+\int_{\mathbb{R}^d}G_T(x)\rho ^u(T,x) \, dx+\int_{\R^d}\langle x\rangle^{2\beta}h(u(x)) \, dx$$
and this implies that $(P)$ and $(P_S)$ are equivalent.

Therefore, we get a double perspective (stochastic and deterministic) on the problem of the optimal displacement of a population. We shall focus in what follows on the problems of proving the 
existence of an optimal control for problem $(P)$ and on deriving Pontryagin's principle. For other recent papers where the investigation of a stochastic optimal control problem is reduced to that of an equivalent deterministic optimal control problem, we refer to \cite{Anita2021}, \cite{Barbu2020}, \cite{BarbuRZ2020}. For certain particular choices of $G, G_T$ and $h$ problem $(P)$ can describe an identification problem related to a Fokker-Planck equation. For some recent papers on this topic we refer to \cite{MarioAlfioJSC}, \cite{MarioAlfioMCA}. For some recent papers concerning the control of non-linear Fokker-Planck equations with the control acting nonlocally on the drift only see \cite{Anita2024_1} and \cite{anita3}. And for a McKean-Vlasov equation with controlled stochasticity we mention \cite{Dipersio}.
\vspace{3mm}

The plan of the paper is as follows:
In Section 2 we will derive the most important properties of the operator $\mathcal{S}$. In Section 3 we prove the equivalence of problems $(P_S)$ and $(P)$ and the existence of an optimal control
for problem $(P)$. Section 4 is devoted to establishing the Pontryagin principle for problem $(P)$.
Section 5 contains some final comments.
\vspace{5mm}

\bf Notations. \rm
We shall use the Einstein summation convention and the following notations:

$B_r(x_0)=B(x_0;r)=\{ x\in \R^d; \, |x-x_0|<r\}, \ B_r=B(0_d;r),$ $m_r=m(B_r)$, where $m$ is the Lebesgue measure,

$A^c=\R^d\setminus A$, 

$L^p$ instead of $L^p(\R^d)$ and $\| \cdot \|_{L^p}=\| \cdot \|_p$ for the usual norm $\| \cdot \|_{L^p(\R^d)}$,
 
$H^p$ instead of $H^p(\R^d)$ and $\| \cdot \|_{H^p}$ for the usual norm $\| \cdot \|_{H^p(\R^d)}$.

Throughout the paper, the notation $C^{1,2}_b$ means that $b$, its first spatial derivatives, its first and second derivatives with respect to the scalar variable are bounded and continuous. Similarly, $C^2_b$ for $q$ means bounded continuous derivatives up to order two in $(x,s)$. The latter convention is adopted to make explicit the coefficient bounds required in the energy estimates and in the global parabolic regularity theorem in the Appendix.

\noindent
We shall denote by the same symbols $C$, $\tilde{C}$ several positive constants appearing in the evaluations.

\section{Properties of the operator $\mathcal{S}$}

Note first that $(\mathcal{S}u)(x)=(\cB * u)(x)=\int_{\R^d}\cB (x-y)u(y) \, dy$, where $\cB$ is the Bessel kernel (we will see that $\cB$ is strictly positive, radially symmetric and satisfies $\int_{\R^d}\cB (x) \, dx = 1$).

$$\mathcal{S}u=\cF^{-1}\left( \frac{1}{1+|\xi |^2} \cF u\right) ,$$
$$\cB (x)=\frac{1}{(4\pi)^{\frac{d}{2}}}\int_0^{\infty}e^{-\theta}\theta^{- {\frac{d}{2}}}e^{-\frac{|x|^2}{4\theta}} d\theta,$$
$\cB \in L^1(\R^d)\cap C^{\infty}(\R^d\setminus \{0\})$, $\cB >0$, $\int_{\R^d}\cB = 1$; here ${\cal F}$ and ${\cal F}^{-1}$ are the Fourier transform and its inverse.

\begin{lemma}\label{lemma2.1}
For any $u\in L^2(\R^d)$ with $0\leq u \leq M_0$ a.e. we get that:
$$0\leq \eta=\mathcal{S}u \leq M_0 \quad \mbox{ a.e. in } \R^d.$$
\end{lemma}

\begin{proof}
We have that:
$$0\leq (\mathcal{S}u)(x)=\int_{\R^d}\cB (x-y)u(y) \, dy \leq M_0 \int_{\R^d}\cB (x-y) dy = M_0$$ 
a.e. $x\in \R^d$.
\end{proof}
\vspace{3mm}

\begin{lemma}\label{lemma2.2}
For any $\beta >0$ there exists a constant $C_{\beta} >0$ such that for any $u\in L^2_{\beta}(\R^d)$ we have that
$$\|\mathcal{S}u\|_{2,\beta}\leq C_{\beta}\|u\|_{2,\beta} \quad \mbox{ and } \quad \|\nabla(\mathcal{S}u)\|_{2,\beta}\leq C_{\beta}\|u\|_{2,\beta}.$$
This implies that $\mathcal{S}:L^2_{\beta}(\R^d)\rightarrow H^1_{\beta}(\R^d)$ is a bounded operator.
\end{lemma}

\begin{proof}
$$\langle x \rangle ^{\beta} |(\mathcal{S}u)(x)|\leq \int_{\R^d}\langle x \rangle ^{\beta} \cB (x-y)|u(y)| \, dy $$
and by Proposition \ref{appendix1} (in Appendix)
$$\leq C_{\beta} \int_{\R^d} \langle x-y \rangle ^{\beta} \cB (x-y) \langle y \rangle ^{\beta} |u(y)| \, dy = C_{\beta}\left( \cB _{\beta}*(\langle y \rangle ^{\beta}|u(y)|)\right)(x)$$
(here $\cB _{\beta} = \langle \cdot \rangle ^{\beta} \cB \in L^1(\R^d)$).
By Young's inequality, we get that
$$\|\mathcal{S}u\|_{2,\beta}\leq C_{\beta}\|\cB _{\beta}\|_{L^1}\|u\|_{2,\beta}.$$
Analogously we obtain that $\|\nabla (\mathcal{S}u)\|_{2,\beta}\leq C_{\beta}\|\langle \cdot \rangle ^{\beta} \nabla \cB\|_{L^1}\|u\|_{2,\beta}$. 
\end{proof}

\noindent
\begin{remark} 
For any $u\in L^{\infty}(\R^d)$ we have that $\mathcal{S}u \in C^1_b(\R^d)$.
\end{remark}

\begin{proof}
For any $x_0,x\in \R^d$:
$$|(\mathcal{S}u)(x)-(\mathcal{S}u)(x_0)|=\left| \int_{\R^d}\left[ \cB (x-y)-\cB (x_0-y)\right]u(y) \, dy\right| $$
$$\leq \| u \|_{L^{\infty}}\int_{\R^d}|\cB (x-y)-\cB (x_0-y)|dy \rightarrow 0 \mbox{ as } x\rightarrow x_0$$ 
(we have used a change of variable $y-x_0=z$ to obtain the convergence).
$$\Rightarrow \mathcal{S}u\in C(\R^d).$$

\noindent
On the other hand, we have:
$$|(\mathcal{S}u)(x)|=\left| \int_{\R^d}\cB (x-y)u(y) dy\right| \leq \|u\|_{L^{\infty}}\int_{\R^d}\cB (x-y) \, dy=\|u \|_{L^{\infty}}$$
$$\Rightarrow \mathcal{S}u\in C_b(\R^d).$$

\noindent
We have $\nabla (\mathcal{S}u)(x)=\int_{\R^d}\nabla \cB (x-y)u(y) \, dy$, where 

$\nabla \cB (x)=C\displaystyle \int^{\infty}_0e^{-\theta}\theta ^{- \frac{d}{2}}e^{-\frac{|x|^2}{4\theta}}\frac{-x}{2\theta} d\theta$ with $C=\displaystyle \frac{1}{(4\pi)^{\frac{d}{2}}}$.

\noindent
So, 
$$\int_{\R^d}|\nabla\cB (x)| dx\leq C\int^{\infty}_0 e^{-\theta}\theta ^{- \frac{d}{2}}\int_{\R^d} e^{-\frac{|x|^2}{4\theta}}\frac{|x|}{2\theta} \, dx \, d\theta .$$ 
We do a change of variable $\frac{x}{2\sqrt{\theta}} =y \Rightarrow dx = (2\sqrt{\theta})^d dy$ and get:
$$\int_{\R^d}|\nabla\cB (x)| \, dx\leq C\int^{\infty}_0 e^{-\theta}\theta ^{- \frac{d}{2}}\int_{\R^d} e^{-|y|^2}|y| \frac{1}{\theta^{\frac{1}{2}}}(2\sqrt{\theta})^d \, dy \, d\theta$$
$$=\tilde{C}\int^{\infty}_0 e^{-\theta}\theta^{-\frac{1}{2}} \,  d\theta \, \int_{\R^d} e^{-|y|^2}|y| \, dy < +\infty$$
$$\Rightarrow \nabla \cB \in L^1(\R^d;\R^d).$$

It follows analogously that $\nabla (\mathcal{S}u) \in C_b(\R^d) \Rightarrow \mathcal{S}u \in C^1_b(\R^d).$
\end{proof}
\vspace{3mm}

\begin{lemma}\label{lemma2.3}
Let $\beta >0$. The following compact embedding holds
$$H^1(\R^d)\cap L^2_{\beta }(\R^d) \hookrightarrow L^2(\R^d) \, . $$
\end{lemma}
\begin{proof}
Let $(f_n)$ be bounded in $H^1(\R^d)\cap L^2_{\beta}(\R^d)$. Let $R>0$ be arbitrary and fixed.
By Rellich-Kondrachov's theorem, we get that $f_n\rightarrow f$ in $L^2(B_R)$ on a subsequence (also denoted by $(f_n)$).

On the other hand, we have:

$$\int_{B_R^c}|f_n|^2 dx \leq R^{-2\beta}\int_{B_R^c}\langle x \rangle ^{2\beta}|f_n|^2 dx \leq R^{-2\beta}\sup_n \|f_n\|^2_{2,\beta}.$$

Thus we get $(f_{1_n})$ a subsequence convergent to $f$ on $L^2(B_1)$, $(f_{2_n})$ a subsequence of the previous sequence which converges to $f$ on $L^2(B_2)$, etc.
We choose $(f_{n_n})$ convergent to $f$ in $L^2(B_R), \forall R>0$.

We have that $\forall \varepsilon >0$, there exists a $R_\varepsilon \in \N^*$ such that $R_\varepsilon ^{-2\beta} \sup_k \|f_k\|^2_{2,\beta} < \frac{\varepsilon^2}{4}$.
So, $\|f_k\|^2_{L^2(B_{R_\varepsilon}^c)}< \frac{\varepsilon^2}{4}$.
Moreover, there exists $n_\varepsilon \in \N$ such that $\forall n \geq n_\varepsilon$ we have that $\|f_{n_n}-f\|^2_{L^2(B_{R_\varepsilon})}<\frac{\varepsilon^2}{4}$. Thus we obtain that $(f_{n_n})$ is a Cauchy sequence in $L^2(\R^d)$ and therefore convergent to $f$ in $L^2(\R^d)$.
\end{proof}
\vspace{2mm}

\begin{lemma}\label{lemma2.4}
Let $d\in \{ 1,2,3\}$ and let $f_{\varepsilon }\in L^{\infty }(\R^d)$ satisfy
$$supp \ f_{\varepsilon }\subset B_{\varepsilon }(x_0), \quad \| f_{\varepsilon }\|_{L^{\infty }(\R^d)}\leq M .$$
Let $\eta _{\varepsilon }=\mathcal{S}f_{\varepsilon }$. Then there exists a constant $C$ independent of $x_0\in \R^d$, $M>0$ and $\varepsilon >0$, such 
that for all sufficiently small $\varepsilon $
$$\| \eta _{\varepsilon }\|^2_{H^1(\R^d)}\leq C \, M^2m_{\varepsilon }^{\frac{3}{2}}.$$
Moreover, if $d=3$, the following sharper estimate holds
$$\| \eta _{\varepsilon }\|^2_{H^1(\R^d)}\leq C \, M^2m_{\varepsilon }^{\frac{5}{3}}.$$
If $d=2$, it holds
$$\| \eta _{\varepsilon }\|^2_{H^1(\R^d)}\leq C \, M^2m_{\varepsilon }^2(1+|\log \varepsilon |),$$
and if $d=1$:
$$\| \eta _{\varepsilon }\|^2_{H^1(\R^d)}\leq C \, M^2m_{\varepsilon }^2.$$
\end{lemma}

\begin{proof}
Since $\eta _{\varepsilon }=\mathcal{S}f_{\varepsilon }$ the weak formulation of $(I-\Delta )\eta _{\varepsilon }=f_{\varepsilon }$ gives
$$\| \eta_{\varepsilon }\|^2_{H^1(\R^d)}=\int_{\R^d}f_{\varepsilon }\eta _{\varepsilon } \, dx=\int_{B_{\varepsilon }(x_0)}
\int_{B_{\varepsilon }(x_0)}f_{\varepsilon }(x)\mathcal{B}(x-y)f_{\varepsilon }(y) \, dy \, dx ,$$
thus
$$\| \eta_{\varepsilon }\|^2_{H^1(\R^d)}\leq M^2\int_{B_{\varepsilon }(x_0)}
\int_{B_{\varepsilon }(x_0)}\mathcal{B}(x-y) \, dy \, dx .$$
The Bessel kernel is translation invariant and satisfies the following local bounds near the origin
$$\mathcal{B}(z)\leq C \quad \mbox{\rm if } d=1, $$
$$\mathcal{B}(z)\leq C(1+|\log |z||) \quad \mbox{\rm if } d=2, $$
$$\mathcal{B}(z)\leq C(1+|z|^{-1}) \quad \mbox{\rm if } d=3 $$
(where $C$ is a positive constant) 
and the exponential decay at infinity is irrelevant, since both variables are integrated over the same small ball.

For $d=1$, the kernel is bounded near the origin, therefore
$$\int_{B_{\varepsilon }(x_0)}\int_{B_{\varepsilon }(x_0)}\mathcal{B}(x-y) \, dy \, dx \leq C \, m^2_{\varepsilon }.$$

For $d=2$, after translating $x_0$ to the origin, we have
$$\int_{B_{\varepsilon }}\int_{B_{\varepsilon }}\mathcal{B}(x-y) \, dy \, dx \leq C \int_{B_{\varepsilon }}\int_{B_{\varepsilon }}(1+|\log |x-y||) \, dy \, dx,$$
where, for each $x\in B_{\varepsilon }$, the inner integral is bounded by
$$C \int_{B_{2\varepsilon }}(1+|\log |z||) \, dz \leq C \varepsilon ^2(1+|\log \varepsilon |).$$
Then we get that 
$$\int_{B_{\varepsilon }}\int_{B_{\varepsilon }}\mathcal{B}(x-y) \, dy \, dx \leq C\varepsilon ^4(1+|\log \varepsilon |)=Cm^2_{\varepsilon }(1+|\log \varepsilon |) .$$

For $d=3$, the same argument gives 
$$\int_{B_{\varepsilon }}\int_{B_{\varepsilon }}\mathcal{B}(x-y) \, dy \, dx\leq C\int_{B_{\varepsilon }}\int_{B_{\varepsilon }}(1+|x-y|^{-1}) \, dy \, dx.$$
For each fixed $x\in B_{\varepsilon }$,
$$\int_{B_{\varepsilon }}|x-y|^{-1} \, dy \leq \int_{B_{2\varepsilon }}|z|^{-1}dz\leq C \, \varepsilon ^2,$$
hence
$$\int_{B_{\varepsilon }}\int_{B_{\varepsilon }}\mathcal{B}(x-y) \, dy \, dx\leq C \, \varepsilon ^3\varepsilon ^2=C \, m^{5/3}_{\varepsilon }.$$
Therefore, for sufficiently small $\varepsilon $, the bounds in dimensions one, two, and three are all bounded by $C \, m^{3/2}_{\varepsilon }$, proving the lemma.

\end{proof}

\vspace{3mm}

\begin{theorem}\label{theorem2.5}
Let $\beta >0$, $u_n, u\in \mathcal{U}_{\beta }$ ($n\in \mathbb{N}^*$) with $u_n\rightharpoonup u$ weakly in $L^2_{\beta }(\R^d)$. Then
$$\eta _n=\mathcal{S}u_n\rightarrow \mathcal{S}u=\eta \quad \mbox{\rm in } H^1(\R^d), \qquad \eta _n\rightarrow \eta \quad \mbox{\rm in } C^1_{loc}(\R^d) \, .$$
\end{theorem}

\begin{proof} ~

\begin{itemize}
\item[\bf Step 1] Since 
$$\| \eta_n\|_{L^2}^2+\| \nabla \eta _n\| _{L^2}^2=\int_{\R^d}u_n\eta _n \, dx \leq \frac{1}{2} \| u_n\| _{L^2}^2+\frac{1}{2}\| \eta _n\|^2_{L^2}$$
it follows that $\{ \eta _n\}_{n\in \mathbb{N}^*}$ is bounded in $H^1(\R^d)$.

On the other hand by Lemma 2.2 we have that $\{ \eta _n\}_{n\in \mathbb{N}^*}$ is bounded in $L^2_{\beta }(\R^d)$ and so we get via Lemma 2.3 that
$$\eta _n \longrightarrow \tilde{\eta } \quad \mbox{\rm strongly in } L^2(\R^d) \ \mbox{\rm and weakly in } H^1(\R^d)$$
(on a subsequence).

\item[\bf Step 2] Since 
$\int_{\R^d}(\eta_n\varphi+\nabla \eta_n\cdot \nabla\varphi )dx=\int_{\R^d}u_n\varphi \,dx$, $\forall \varphi \in H^1(\R^d)$, it follows that 
$$\int_{\R^d}(\tilde{\eta}\varphi+\nabla \tilde{\eta }\cdot \nabla \varphi )dx=\int_{\R^d}u\varphi \,dx, \ \forall \varphi \in H^1(\R^d) \Rightarrow \tilde{\eta }=\mathcal{S}u=\eta \, .$$

\item[\bf Step 3] We have that $\| \eta _n-\eta \|^2_{H^1}=\int_{\R^d}(u_n-u)(\eta_n-\eta) \, dx \longrightarrow 0$ (because $\{ u_n-u\}_{n\in \mathbb{N}^*}$ is bounded in $L^2(\R^d)$
and $\{ \eta_n-\eta \}_{n\in \mathbb{N}^*}$ converges to $0$ in $L^2(\R^d)$) and so we get that
$\eta _n\longrightarrow \eta $ in $H^1(\R^d)$ on a subsequence.

Arguing by contradiction, we conclude that 
$$\eta _n\longrightarrow \eta \quad \mbox{\rm in } H^1(\R^d) \quad \mbox{\rm (the whole sequence)} .$$

\item[\bf Step 4] Consider an arbitrary but fixed $R>0$. Fix the ball $B_{2R}$.
Since $\|u_n\|_{L^{\infty }(B_{2R})}\leq M_0, \ \forall n\in \mathbb{N}^*$ we get via the interior elliptic regularity
that 
$$\| \eta _n\| _{W^{2,q}(B_R)} \leq C_{R,q}, \quad \forall q>d$$
($C_{R,q}$ is a positive constant depending only on $R, q$ and $M_0$). By the Sobolev embedding, we conclude that
$$\eta _n\longrightarrow \bar{\eta}=\eta \quad \mbox{\rm in } C^1(\overline{B_R}) \ \mbox{\rm on a subsequence}.$$
Arguing by contradiction, we conclude that the whole sequence satisfies
$$\eta _n\longrightarrow \eta \quad \mbox{\rm in } C^1(\overline{B_R}), \ \forall R>0.$$
\end{itemize}
\end{proof}

\section{Equivalence of problems $(P_S)$ and $(P)$ and existence of an optimal control}

\begin{theorem}
For any $u\in\cU_{\beta}$ there exists a unique weak solution $\rho^u \in L^2(0,T;H^1(\R^d))\cap W^{1,2}(0,T;H^{-1}(\R^d))$ for
\begin{equation*}
\left\{ \begin{array}{ll}
\displaystyle \frac{\partial\rho}{\partial t}(t,x)=-\nabla\cdot (b(x,(\mathcal{S}u)(x))\rho(t,x))+\frac{1}{2}\frac{\partial ^2}{\partial x_i \partial x_j}(q_{ij}(x,(\mathcal{S}u)(x))\rho(t,x)), & (t,x)\in (0,T)\times \R^d, \\
\rho(0,x)=\rho_0(x), & x\in \R^d.
\end{array} \right.
\end{equation*}
\end{theorem}
\begin{proof}
Indeed, let $a: H^1(\R^d)\times H^1(\R^d) \rightarrow \R$ given by:
$$a(\phi,\psi)=-\int_{\R^d}b(x,(\mathcal{S}u)(x)) \phi (x) \cdot \nabla \psi (x)dx+\frac{1}{2}\int_{\R^d}\frac{\partial}{\partial x_j}(q_{ij}(x,(\mathcal{S}u)(x))\phi (x))\frac{\partial \psi}{\partial x_i}(x)dx$$
$$=-\int_{\R^d}b(x,(\mathcal{S}u)(x)) \phi (x) \cdot \nabla \psi (x)dx+\frac{1}{2}\int_{\R^d}q_{ij}(x,(\mathcal{S}u)(x))\frac{\partial \phi}{\partial x_j}(x)\frac{\partial \psi}{\partial x_i}(x)dx$$
$$+\frac{1}{2}\int_{\R^d}\left[ \frac{\partial q_{ij}}{\partial x_j} (x,(\mathcal{S}u)(x))+\frac{\partial q_{ij}}{\partial s}(x,(\mathcal{S}u)(x))\frac{\partial (\mathcal{S}u)}{\partial x_j}(x) \right] \phi (x)\frac{\partial \psi}{\partial x_i}(x)dx.$$

The properties of $q_{ij}$ and the fact that $\mathcal{S}u \in C^1_b(\R^d)$ imply that:
\begin{itemize}
\item $a$ is bilinear;
\item $a$ is bounded, i.e. $\exists M>0$: $|a(\phi , \psi)|\leq M\|\phi \|_{H^1}\|\psi \|_{H^1}$, $\forall \phi , \psi \in H^1(\R^d)$;
\item there exist $\alpha _0>0$, $\nu _0 \in \R$ such that:
$$a(\phi ,\phi )\geq \alpha _0 \|\phi \|^2_{H^1}-\nu _0\|\phi\|^2_{L^2}, \quad \forall \phi \in H^1(\R^d).$$
\end{itemize}
We may conclude (via Lions' theorem - see \cite{Brezis2005}, \cite{Lions}) that there exists a unique weak solution $\rho ^u$.
\end{proof}

\begin{remark} Moreover, by multiplying the Fokker-Planck equation by $(\rho ^u)^-$ we get after a standard evaluation that $\rho ^u(t,x)\geq 0$ a.e. $x \in \R^d, \forall t \in [0,T]$.

For any $u\in {\cal U}_{\beta}$, it follows as in \cite{Anita2021} that $\rho ^u$ is also a distributional solution to (\ref{FPeq}), that $\rho ^u\in C([0,T];L^1(\mathbb{R}^d))$ (and so, it is t-narrowly continuous) and that 
$\rho ^u(t)$ is a probability density for any $t\in [0,T]$. By the superposition principle (see \cite{Figalli2008}, \cite{Huang2021}, \cite{Trevisan2016}) we get that there exists a unique (in law) probabilistically
weak solution $(\tilde{X}^u(t))_{t\in [0,T]}$ to (\ref{SDE}) such that $\rho ^u(t)$ is the probability density of $\tilde{X}^u(t)$ for any $t\in [0,T]$.  
\end{remark}

Assume that $X_0$ has the probability density $\rho _0$ and that $\E[|X_0|^2]<\infty$, then for any $u\in {\cal U}_{\beta }$, problem (\ref{SDE}) admits a unique (strong) solution $(X^u(t))_{t\in [0,T]}$.

By Theorem 1.3 in \cite{RoecknerZhang2010} we get that $\rho ^u(t)$ is the probability density for $X^u(t)$ for any $t\in [0,T]$ as well. It follows that
$$\mathbb{E}\Big[\int_0^T G(t,X^u(t)) \, dt+G_T(X^u(T)) \Big]+\int_{\R^d}\langle x\rangle^{2\beta}h(u(x)) \, dx$$
$$=\int_0^T \int_{\mathbb{R}^d}G(t,x)\rho ^u(t,x) \, dx \, dt+\int_{\mathbb{R}^d}G_T(x)\rho ^u(T,x) \, dx+\int_{\R^d}\langle x\rangle^{2\beta}h(u(x)) \, dx.$$
We conclude that $(P)$ and $(P_S)$ are equivalent, with $X^u$ being the (strong) solution to (\ref{SDE}).

\vspace{3mm}

\begin{proposition}\label{proposition3.4}
Let $u_n,u\in\cU_\beta$ and assume that $u_n\rightharpoonup u$ weakly in $L^2_\beta(\R^d)$. Then
$$\rho^{u_n}\longrightarrow \rho^u \quad \mbox{in } C([0,T];L^2(\R^d)),$$
and
$$\rho^{u_n}\longrightarrow \rho^u \quad \mbox{in } L^2(0,T;H^1(\R^d)).$$
In particular, $\rho^{u_n}(T)\to \rho^u(T)$ strongly in $L^2(\R^d)$.
\end{proposition}

To simplify the notation, when no confusion is possible, we write $F(\mathcal{S}u)$ instead of $F(x,(\mathcal{S}u)(x))$.

\begin{proof}
Let
\[
\eta_n:=\mathcal S u_n,\qquad \eta:=\mathcal S u.
\]
By Theorem \ref{theorem2.5}, $\eta_n\to \eta$ in $H^1(\R^d)$ and in $C^1_{\rm loc}(\R^d)$. Since $0\le u_n,u\le M_0$ and $\nabla\mathcal B\in L^1(\R^d;\R^d)$, the sequence $\nabla \eta_n$ is uniformly bounded in $L^\infty(\R^d;\R^d)$. Hence
\[
 b(\eta_n),\quad q(\eta_n),\quad \nabla_x q(\eta_n),\quad \frac{\partial q}{\partial s}(\eta_n),\quad \frac{\partial q}{\partial s}(\eta_n)\nabla\eta_n
\]
converge pointwise a.e. to the corresponding quantities evaluated at $\eta$, and they are uniformly bounded.

Set $z_n:=\rho^{u_n}-\rho^u$. We write
\[
B_n=b(\eta_n),\qquad B=b(\eta),\qquad Q_n=q(\eta_n),\qquad Q=q(\eta),
\]
and
\[
D_{i,n}:=\frac{\partial Q_{ij,n}}{\partial x_j}
=\frac{\partial q_{ij}}{\partial x_j}(\eta_n)+\frac{\partial q_{ij}}{\partial s}(\eta_n)\frac{\partial \eta_n}{\partial x_j},\qquad
D_i:=\frac{\partial Q_{ij}}{\partial x_j}.
\]
The weak formulation of the difference equation gives, for every $\varphi\in H^1(\R^d)$,
\begin{align*}
\left\langle \frac{d z_n}{dt},\varphi\right\rangle_{H^{-1},H^1}
&=\int_{\R^d}B_nz_n\cdot\nabla\varphi \,dx
-\frac12\int_{\R^d}Q_{ij,n}\frac{\partial z_n}{\partial x_j}\frac{\partial \varphi}{\partial x_i}\,dx
-\frac12\int_{\R^d}D_{i,n}z_n\frac{\partial \varphi}{\partial x_i}\,dx
\\
&\quad+\int_{\R^d}(B_n-B)\rho^u\cdot\nabla\varphi\,dx
-\frac12\int_{\R^d}(Q_{ij,n}-Q_{ij})\frac{\partial \rho^u}{\partial x_j}\frac{\partial \varphi}{\partial x_i}\,dx
\\
&\quad-\frac12\int_{\R^d}(D_{i,n}-D_i)\rho^u\frac{\partial \varphi}{\partial x_i}\,dx .
\end{align*}
Taking $\varphi=z_n(t)$, using the uniform ellipticity of $q$, the boundedness of the lower-order coefficients, and Young's inequality, we obtain
\begin{align*}
\frac12\frac{d}{dt}\|z_n(t)\|_{L^2}^2+c_0\|\nabla z_n(t)\|_{L^2}^2
&\le C\Big( \|z_n(t)\|_{L^2}^2
+\|(B_n-B)\rho^u(t)\|_{L^2}^2\\
&\quad+\|(Q_n-Q)\nabla\rho^u(t)\|_{L^2}^2
+\|(D_n-D)\rho^u(t)\|_{L^2}^2\Big),
\end{align*}
where $c_0>0$ and $C>0$ are independent of $n$. Since $z_n(0)=0$, Gronwall's lemma yields
\begin{align*}
\sup_{t\in[0,T]}\|z_n(t)\|_{L^2}^2+\int_0^T\|\nabla z_n(t)\|_{L^2}^2\,dt
&\le C\Big[ \int_0^T\|(B_n-B)\rho^u(t)\|_{L^2}^2\,dt\\
&\quad+\int_0^T\|(Q_n-Q)\nabla\rho^u(t)\|_{L^2}^2\,dt\\
&\quad+\int_0^T\|(D_n-D)\rho^u(t)\|_{L^2}^2\,dt\Big].
\end{align*}
The three terms on the right-hand side tend to zero by dominated convergence. Indeed, the coefficient differences converge pointwise a.e. and are uniformly bounded, whereas $\rho^u\in L^2((0,T)\times\R^d)$ and $\nabla\rho^u\in L^2((0,T)\times\R^d;\R^d)$, hence
\[
\sup_{t\in[0,T]}\|z_n(t)\|_{L^2}\to0,\qquad
\int_0^T\|\nabla z_n(t)\|_{L^2}^2\,dt\to0\,,
\]
which proves the asserted convergence in $C([0,T];L^2(\R^d))$ and in $L^2(0,T;H^1(\R^d))$.
\end{proof}
\vspace{2mm}

\subsection{Existence of an optimal control}

We assume, in addition to hypotheses (H1)-(H4), that
\begin{itemize}
\item[\bf (H4)'] there exists $\alpha>0$ such that
$$h(s)\geq \frac{\alpha}{2}s^2,\quad \forall s\in[0,M_0].$$
\end{itemize}

\begin{theorem}\label{existence_optimal_control}
Under hypotheses (H1)-(H4) and (H4)', problem $(P)$ admits at least one optimal control $u^*\in\cU_\beta$.
\end{theorem}

\begin{proof}
Because $h(0)=0$ and $G,G_T\geq0$, the zero control is admissible and has finite cost. Hence
$$m:=\inf_{u\in\cU_\beta}J(u)$$
is finite and non-negative. Let $(u_n)$ be a minimizing sequence such that $J(u_n)\to m$. From (H4)' we obtain
$$\frac{\alpha}{2}\int_{\R^d}\langle x\rangle^{2\beta}|u_n(x)|^2\,dx\leq J(u_n)\leq m+1$$
for all sufficiently large $n$. Thus $(u_n)$ is bounded in $L^2_\beta(\R^d)$. It is also bounded in $L^\infty(\R^d)$ and satisfies $0\leq u_n\leq M_0$ a.e. By Banach-Alaoglu's theorem and weak compactness in the Hilbert space $L^2_\beta(\R^d)$, there exists $u^*\in L^2_\beta(\R^d)\cap L^\infty(\R^d)$ such that, up to a subsequence,
$$u_n\rightharpoonup u^*\quad \mbox{weakly in }L^2_\beta(\R^d),\qquad u_n\stackrel{*}{\rightharpoonup}u^*\quad \mbox{weakly * in }L^\infty(\R^d).$$
Therefore, $u^*\in\cU_\beta$.

By Theorem \ref{theorem2.5}, $\mathcal{S}u_n\to\mathcal{S}u^*$ in $H^1(\R^d)$ and in $C^1_{loc}(\R^d)$. By Proposition \ref{proposition3.4},
$$\rho^{u_n}\to\rho^{u^*}\quad \mbox{in }L^2((0,T)\times\R^d),\qquad \rho^{u_n}(T)\to\rho^{u^*}(T)\quad \mbox{in }L^2(\R^d).$$
Since $G\in L^2((0,T)\times\R^d)$ and $G_T\in L^2(\R^d)$, the two state-dependent terms in $J$ converge strongly by duality. Finally, the map
$$u\mapsto \int_{\R^d}\langle x\rangle^{2\beta}h(u(x))\,dx$$
is weakly lower semicontinuous on $L^2_\beta(\R^d)$, because $h$ is convex and continuous on $[0,M_0]$. Hence
$$J(u^*)\leq \liminf_{n\to\infty}J(u_n)=m.$$
Since $u^*\in\cU_\beta$, we also have $m\leq J(u^*)$, and therefore $J(u^*)=m$.
\end{proof}

\section{The Pontryagin maximum principle}
\label{sec:PMP}

Assume that besides (H1)-(H4) the following assumption holds
\begin{itemize}
\item[\bf (H5)] $\rho_0, \, G_T\in H^1(\R^d)$, $d\in \{ 1,2,3\}$ and $q_{ij}=q \, \delta_{ij}$, where $\delta _{ij}$ is the symbol of Kronecker.
\end{itemize}

Consider the associated dual problem
\begin{equation}\label{eqp}
\left\{ \begin{array}{ll}
\displaystyle \frac{\partial p}{\partial t}=-b(\mathcal{S}u^*)\cdot \nabla p-\frac{1}{2}q(\mathcal{S}u^*)\Delta p-G, \quad &(t,x)\in (0,T)\times \R^d ,\\
p(T,x)=G_T(x), & x\in \R^d ,
\end{array} \right.
\end{equation}
where $u^*$ is an optimal control for the deterministic problem $(P)$. Denote $\rho = \rho ^{u^*}$.

By Proposition \ref{appendix_parabolic_regularity} we get that the weak solutions to the state equation associated with $u^*$ and to the backward adjoint equation \eqref{eqp} admit the regularity
$$\rho,\ p\in L^2(0,T;H^2(\R^d))\cap W^{1,2}(0,T;L^2(\R^d)).$$

\begin{theorem}\label{theorem4.1} (Pontryagin's principle)
Let $u^*$ be an optimal control and let $p$ be the weak solution to \eqref{eqp}. Define
$$F_1(x):=\int_0^T \frac{\partial b}{\partial s}(x,(\mathcal{S}u^*)(x))\rho(t,x)\cdot\nabla p(t,x)\,dt,$$
and
$$F_2(x):=\int_0^T \frac{\partial q}{\partial s}(x,(\mathcal{S}u^*)(x))\rho(t,x)\Delta p(t,x)\,dt.$$
Then $F_1,F_2\in L^1(\R^d)$. Set
$$\Phi_1:=\mathcal{S}F_1=\cB*F_1,
\qquad
\Phi_2:=\mathcal{S}F_2=\cB*F_2.$$
For a.e. $x\in\R^d$ one has
$$u^*(x)\in\argmin_{w\in[0,M_0]}\left\{w\Phi_1(x)+\frac12 w\Phi_2(x)+\langle x\rangle^{2\beta}h(w)\right\}.$$
If the function $h$ is strictly convex, the minimizer is unique for a.e. $x$.
\end{theorem}

\begin{proof}
The proof is carried out at an arbitrary point $x_0$ which is simultaneously a Lebesgue point of $u^*$, $h\circ u^*$, $\Phi_1$, and $\Phi_2$, which is enough, since these functions are locally integrable and the set of such points has full Lebesgue measure.

Fix $w\in[0,M_0]$ and define the spike perturbation
$$u_\varepsilon(x)=w\mathds{1}_{B_\varepsilon(x_0)}(x)+u^*(x)\mathds{1}_{B_\varepsilon(x_0)^c}(x),
\qquad \delta_\varepsilon:=u_\varepsilon-u^*=(w-u^*)\mathds{1}_{B_\varepsilon(x_0)}.$$
Let $\rho_\varepsilon:=\rho^{u_\varepsilon}$ and $z_\varepsilon:=\rho_\varepsilon-\rho$. Since $u^*$ is optimal
\begin{equation}\label{ineqzeps_blue}
0\leq J(u_\varepsilon)-J(u^*)
=\int_0^T\int_{\R^d}Gz_\varepsilon\,dx \, dt+
\int_{\R^d}G_Tz_\varepsilon(T)\,dx+
\int_{B_\varepsilon(x_0)}\langle x\rangle^{2\beta}(h(w)-h(u^*(x)))\,dx.
\end{equation}
The last term satisfies
\begin{equation}\label{control_term_blue}
\frac1{m_\varepsilon}\int_{B_\varepsilon(x_0)}\langle x\rangle^{2\beta}(h(w)-h(u^*(x)))\,dx
\longrightarrow
\langle x_0\rangle^{2\beta}(h(w)-h(u^*(x_0))).
\end{equation}

The duality identity obtained by multiplying the equation for $z_\varepsilon$ by the adjoint $p$ and integrating by parts is (after an easy calculation)
\begin{align}\label{duality_blue}
&\int_0^T\int_{\R^d}Gz_\varepsilon\,dx \, dt+
\int_{\R^d}G_Tz_\varepsilon(T)\,dx \\
&=\int_0^T\int_{\R^d}(b(\mathcal{S}u_\varepsilon)-b(\mathcal{S}u^*))\rho\cdot\nabla p\,dx \,dt
+\frac12\int_0^T\int_{\R^d}(q(\mathcal{S}u_\varepsilon)-q(\mathcal{S}u^*))\rho\Delta p\,dx \, dt \nonumber\\
&\quad +\int_0^T\int_{\R^d}(b(\mathcal{S}u_\varepsilon)-b(\mathcal{S}u^*))z_\varepsilon\cdot\nabla p\,dx \, dt
+\frac12\int_0^T\int_{\R^d}(q(\mathcal{S}u_\varepsilon)-q(\mathcal{S}u^*))z_\varepsilon\Delta p\,dx \, dt. \nonumber
\end{align}
The integrations by parts are justified by the regularity quoted before the theorem.

We first estimate the state increment. In particular, the difference $z_\varepsilon$ solves the equation obtained by subtracting the two Fokker-Planck equations, and the forcing terms are produced by the coefficient differences
\[
 b(\mathcal S u_\varepsilon)-b(\mathcal S u^*),\qquad
 q(\mathcal S u_\varepsilon)-q(\mathcal S u^*),\qquad
 \nabla \left(q(\mathcal S u_\varepsilon)\right)-\nabla (q(\mathcal S u^*))\,,
\]
hence, by the mean value theorem and the boundedness assumptions on the coefficients, each of these differences is bounded by a constant times either $|\mathcal S\delta_\varepsilon|$ or $|\nabla (\mathcal S\delta_\varepsilon)|$, and Lemma \ref{lemma2.4} gives
\[
\|\mathcal S\delta_\varepsilon\|_{H^1(\R^d)}\le C m_\varepsilon^{3/4},
\]
for all sufficiently small $\varepsilon$. Since $d\le3$, the Sobolev embedding $H^1(\R^d)\subset L^4(\R^d)$ also gives
\[
\|\mathcal S\delta_\varepsilon\|_{L^4(\R^d)}\le C m_\varepsilon^{3/4}.
\]
Testing the equation for $z_\varepsilon$ by $z_\varepsilon$, using the uniform ellipticity of $q$, and applying Young's inequality, we obtain
\begin{align*}
\|z_\varepsilon\|_{L^\infty(0,T;L^2)}^2+\|z_\varepsilon\|_{L^2(0,T;H^1)}^2
&\le C\|\mathcal S\delta_\varepsilon\|_{L^4}^2
\left(\|\rho\|_{L^2(0,T;L^4)}^2+\|\nabla\rho\|_{L^2(0,T;L^4)}^2\right)
\\
&\quad+C\|\nabla(\mathcal S\delta_\varepsilon )\|_{L^2}^2\|\rho\|_{L^2(0,T;L^\infty)}^2 .
\end{align*}
The right-hand side is bounded by $C m_\varepsilon^{3/2}$, because $\rho\in L^2(0,T;H^2(\R^d))$ and $H^2(\R^d)\subset W^{1,4}(\R^d)\cap L^\infty(\R^d)$ for $d\le3$, hence
\begin{equation*}
\|z_\varepsilon\|_{L^\infty(0,T;L^2(\R^d))}+\|z_\varepsilon\|_{L^2(0,T;H^1(\R^d))}\le C m_\varepsilon^{3/4}\,.
\end{equation*}
Consequently,
\begin{equation}\label{z_terms_vanish_blue}
\frac1{m_\varepsilon}\left|\int_0^T\int_{\R^d}(b(\mathcal{S}u_\varepsilon)-b(\mathcal{S}u^*))z_\varepsilon\cdot\nabla p\,dx\ dt\right|\to0\,,
\end{equation}
and
\begin{equation}\label{z_terms_vanish_q_blue}
\frac1{m_\varepsilon}\left|\int_0^T\int_{\R^d}(q(\mathcal{S}u_\varepsilon)-q(\mathcal{S}u^*))z_\varepsilon\Delta p\,dx\ dt\right|\to0.
\end{equation}
Indeed, for the drift term,
\begin{align*}
\frac1{m_\varepsilon}\left|\int_0^T\int_{\R^d}(b(\mathcal S u_\varepsilon)-b(\mathcal S u^*))z_\varepsilon\cdot\nabla p\,dx\ dt\right|
&\le \frac{C}{m_\varepsilon}\|\mathcal S\delta_\varepsilon\|_{L^4}\|z_\varepsilon\|_{L^2(0,T;L^4)}\|\nabla p\|_{L^2(0,T;L^2)}\to0,
\end{align*}
and the diffusion term is analogous, with $\Delta p$ in place of $\nabla p$.

We now identify the two first-order terms. Since $\rho\in L^2((0,T)\times\R^d)$, $\nabla p\in L^2((0,T)\times\R^d)$, and $\Delta p\in L^2((0,T)\times\R^d)$, while $\frac{\partial b}{\partial s}$ and $\frac{\partial q}{\partial s}$ are bounded, we have $F_1,F_2\in L^1(\R^d)$. The Bessel kernel $\cB$ belongs to $L^1(\R^d)$; hence $\Phi_1=\cB*F_1$ and $\Phi_2=\cB*F_2$ belong to $L^1(\R^d)$ and are defined a.e.

By the mean value theorem,
$$b(\mathcal{S}u_\varepsilon)-b(\mathcal{S}u^*)=\frac{\partial b}{\partial s}(\mathcal{S}u^*)\mathcal{S}\delta_\varepsilon+R^b_\varepsilon,$$
where
$$|R^b_\varepsilon|\leq C|\mathcal{S}\delta_\varepsilon|^2.$$
Therefore,
\begin{equation*}
\int_0^T\int_{\R^d}(b(\mathcal{S}u_\varepsilon)-b(\mathcal{S}u^*))\rho\cdot\nabla p\,dx \, dt
=\int_{\R^d}F_1(x)(\mathcal{S}\delta_\varepsilon)(x)\,dx+o(m_\varepsilon)\,,
\end{equation*}
\noindent
and the remainder is $o(m_\varepsilon)$: since $d\le3$, $H^2(\R^d)\subset W^{1,4}(\R^d)$ and
\[
\frac1{m_\varepsilon}\int_0^T\int_{\R^d}|\mathcal{S}\delta_\varepsilon|^2|\rho||\nabla p|\,dx\,dt
\le
\frac{\|\mathcal{S}\delta_\varepsilon\|_{L^4}^2}{m_\varepsilon}
\|\rho\|_{L^2(0,T;L^4)}\|\nabla p\|_{L^2(0,T;L^4)}
\le C m_\varepsilon^{1/2}\to0\,,
\]
\noindent
then, using the self-adjointness of $\mathcal{S}$, or equivalently Fubini's theorem and the symmetry of $\cB$,
$$\int_{\R^d}F_1(x)(\mathcal{S}\delta_\varepsilon)(x)\,dx
=\int_{B_\varepsilon(x_0)}(w-u^*(y))\Phi_1(y)\,dy.$$
Since $x_0$ is a Lebesgue point of $u^*$ and $\Phi_1$,
\begin{equation}\label{drift_limit_blue}
\frac1{m_\varepsilon}\int_0^T\int_{\R^d}(b(\mathcal{S}u_\varepsilon)-b(\mathcal{S}u^*))\rho\cdot\nabla p\,dx \, dt
\to (w-u^*(x_0))\Phi_1(x_0).
\end{equation}

The diffusion term is identical. We write
$$q(\mathcal{S}u_\varepsilon)-q(\mathcal{S}u^*)=\frac{\partial q}{\partial s}(\mathcal{S}u^*)\mathcal{S}\delta_\varepsilon+R^q_\varepsilon,
\qquad |R^q_\varepsilon|\leq C|\mathcal{S}\delta_\varepsilon|^2.$$
Then
\begin{equation}\label{diffusion_limit_blue}
\frac1{m_\varepsilon}\int_0^T\int_{\R^d}(q(\mathcal{S}u_\varepsilon)-q(\mathcal{S}u^*))\rho\Delta p\,dx \, dt
\to (w-u^*(x_0))\Phi_2(x_0).
\end{equation}
The proof is the same as above, with $F_2$ and $\Phi_2$ in place of $F_1$ and $\Phi_1$.

Dividing \eqref{ineqzeps_blue} by $m_\varepsilon$ and using \eqref{control_term_blue}, \eqref{duality_blue}, \eqref{z_terms_vanish_blue}, \eqref{z_terms_vanish_q_blue}, \eqref{drift_limit_blue}, and \eqref{diffusion_limit_blue}, we obtain
$$0\leq \langle x_0\rangle^{2\beta}(h(w)-h(u^*(x_0)))+(w-u^*(x_0))\Phi_1(x_0)+\frac12(w-u^*(x_0))\Phi_2(x_0).$$
Equivalently,
$$\langle x_0\rangle^{2\beta}h(u^*(x_0))+u^*(x_0)\Phi_1(x_0)+\frac12u^*(x_0)\Phi_2(x_0)
\leq
\langle x_0\rangle^{2\beta}h(w)+w\Phi_1(x_0)+\frac12w\Phi_2(x_0)$$
for every $w\in[0,M_0]$. Since $x_0$ was an arbitrary common Lebesgue point, the conclusion follows for a.e. $x\in\R^d$.
\end{proof}

\begin{remark}
For some classic works concerning the spike-perturbation technique see \cite{Barbu1994} and \cite{Sumin2009}.
\end{remark}

\section{Final comments}

An interesting control problem concerns the situation when instead of the set of controllers ${\cal U}_{\beta }$ we consider the set of controllers 
$${\cal U} =\{ v\in L^{\infty }(\R^d); \, 0\leq v(x)\leq M_0 \ \mbox{\rm a.e. } x\in \R^d, \ v(x)=0 \ \mbox{\rm if } x\in \omega ^c\} ,$$
where $\omega $ is an open and bounded subset of $\R^d$ and the problem to be investigated is
$$\underset{u \in {\cal U}}{\text{Minimize}} \  \int_0^T \int_{\mathbb{R}^d}G(t,x)\rho ^u(t,x) \, dx \, dt+\int_{\mathbb{R}^d}G_T(x)\rho ^u(T,x) \, dx
+\int_{\mathbb{R}^d}h(x,u(x)) \, dx. \leqno({\bf P}0)$$
Note that ${\cal U}$ is a convex and closed (in $L^2_{\beta}(\R^d)$) subset of ${\cal U}_{\beta }$.
\vspace{3mm}

It is worth noting that the techniques developed in the present paper can be extended to more general assumptions on $h$ and to non-scalar diffusion matrices, provided the corresponding global Sobolev regularity and duality estimates are established in the appropriate matrix-valued setting.
\vspace{3mm}

A possible continuation of the present work is the treatment of time-dependent controls \(u=u(t,x)\), with elliptic smoothing performed at each time level through \(\eta(t,\cdot)=\mathcal{S}u(t,\cdot)\), which would require compactness in both time and space for the admissible controls, and the Pontryagin principle would yield a pointwise minimisation rule in \((t,x)\). Moreover, a further substantial extension is the treatment of matrix-valued non-scalar diffusions, since, in that case, the adjoint equation involves the full Hessian \(D^2p\), and the Fokker-Planck equation contains second-order derivatives of \(q_{ij}\rho\), hence requiring a separate global Sobolev regularity analysis for the non-divergence and divergence formulations with matrix-valued coefficients.

\vspace{3mm}
A further direction is to replace the resolvent $\mathcal S=(I-\Delta)^{-1}$ by fractional elliptic smoothing operators $(I-\Delta)^{-\theta}$, $\theta>0$. Indeed, in the latter setting, the spike analysis then depends on the exact singular behaviour of the fractional Bessel kernel at the origin. Identifying the sharp range of $\theta$ for which the weak spike asymptotics still yields a Pontryagin principle would connect the present model with a broader class of nonlocal interaction mechanisms.

\section{Appendix}

\begin{proposition}\label{appendix1}
For a fixed $\beta >0$ we have that there exists $C_{\beta } >0$ such that
$$\langle x \rangle ^{\beta} \leq C_{\beta}\langle x-y \rangle ^{\beta} \langle y \rangle ^{\beta}, \quad \forall x, y \in \R^d .$$
\end{proposition}

\begin{proof}
Indeed, $1+|x|^2=1+|x-y+y|^2\leq 2(1+|x-y|^2)(1+|y|^2)$ and so we get the conclusion.
\end{proof}

\begin{proposition}\label{appendix_parabolic_regularity}
Let $T>0$ and let $\tilde a = (\tilde a_{ij})_{i,j\in \{1,2,...,d\}}, \tilde b =(\tilde b_i )_{i\in \{1,2,...,d\}}, \tilde c$ be time-independent coefficients on $\R^d$. Assume that
\[
\tilde a=\tilde a^T,
\qquad
\tilde a\in W^{1,\infty}(\R^d;\R^{d\times d}),
\qquad
\tilde b\in L^\infty(\R^d;\R^d),
\qquad
\tilde c\in L^\infty(\R^d)\,,
\]
and that there exists a constant \(0<\lambda\) such that
\[
\lambda |\xi|^2\le \tilde a_{ij}(x)\xi_i\xi_j\ ,
\qquad\forall x\in\R^d\,,
\quad\forall \xi\in\R^d\,.
\]
If \(f\in L^2((0,T)\times\R^d)\) and \(v_0\in H^1(\R^d)\), then the Cauchy problem
$$\frac{\partial v}{\partial t}-\tilde a_{ij}\frac{\partial^2 v}{\partial x_i \partial x_j}-\tilde b_i\frac{\partial v}{\partial x_i}-\tilde cv=f,
\qquad v(0)=v_0,$$
has a unique solution satisfying
\[
v\in L^2(0,T;H^2(\R^d))\cap W^{1,2}(0,T;L^2(\R^d)).
\]
Moreover,
\[
\|v\|_{L^2(0,T;H^2(\R^d))}+
\|\frac{\partial v}{\partial t}\|_{L^2((0,T)\times\R^d)}
\le
C\left(
\|f\|_{L^2((0,T)\times\R^d)}+
\|v_0\|_{H^1(\R^d)}
\right),
\]
where \(C\) depends only on
\[
T,d,\lambda,
\|\tilde a\|_{W^{1,\infty}},
\|\tilde b\|_{L^\infty},
\|\tilde c\|_{L^\infty}.
\]
The analogous backward terminal-value problem follows from the change of variables \(s=T-t\).
\end{proposition}

\begin{proof}
This is the standard global Sobolev regularity theorem for uniformly parabolic equations on the whole space with Lipschitz leading coefficients and bounded lower-order coefficients. More precisely, the statement follows from the classical $L^2$-theory for non-divergence form parabolic operators with time-independent coefficients, applied to
\[
\widetilde L v:=\widetilde a_{ij}(x)\frac{\partial^2 v}{\partial x_i\partial x_j}+\widetilde b_i(x)\frac{\partial v}{\partial x_i}+\widetilde c(x)v.
\]
The assumptions $\widetilde a=\widetilde a^T$, $\widetilde a\in W^{1,\infty}(\R^d)$, and uniform ellipticity imply the elliptic estimate
\[
\|v\|_{H^2(\R^d)}\le C\left(\|\widetilde L v\|_{L^2(\R^d)}+\|v\|_{L^2(\R^d)}\right),\qquad v\in H^2(\R^d),
\]
where the lower-order terms are controlled by $\|\widetilde b\|_{L^\infty}$ and $\|\widetilde c\|_{L^\infty}$. Combining this elliptic estimate with the standard energy estimate for $\displaystyle \frac{\partial v}{\partial t}-\widetilde L v=f$ gives the asserted maximal $L^2$-regularity estimate. The result is stated, in equivalent forms, in Ladyzhenskaya-Solonnikov-Ural'tseva \cite[Chapter IV]{Ladyzhenskaya1968}, Krylov \cite[Chapters 2 and 3]{Krylov2008}, and Lieberman \cite[Chapters VII and VIII]{Lieberman1996}. The terminal-value problem is reduced to the initial-value problem by the change of variables $s=T-t$.
\end{proof}

We record explicitly why Proposition \ref{appendix_parabolic_regularity} applies in the present paper. For the adjoint equation \eqref{eqp}, after the change of variables $r(s,x)=p(T-s,x)$, one has
\[
\widetilde a(x)=\frac12 q(x,\mathcal S u^*(x))I,
\qquad
\widetilde b(x)=b(x,\mathcal S u^*(x)),
\qquad
\widetilde c(x)=0.
\]
The uniform ellipticity follows from (H2), and $\widetilde a\in W^{1,\infty}(\R^d)$ follows from (H2), $\mathcal S u^*\in W^{1,\infty}(\R^d)$, and the identity
\[
\nabla \bigl(q(x,\mathcal S u^*(x))\bigr)=
\nabla_x q(x,\mathcal S u^*(x))+\frac{\partial q}{\partial s}(x,\mathcal S u^*(x))
\nabla(\mathcal S u^*)(x).
\]
For the Fokker-Planck equation in the scalar case $q_{ij}=q\delta_{ij}$, rewriting the equation in non-divergence form gives
\[
\frac{\partial \rho}{\partial t}-\frac12 q(\mathcal S u^*)\Delta\rho-\bigl(\nabla (q(\mathcal S u^*))-b(\mathcal S u^*)\bigr)\cdot\nabla\rho-
\left(\frac12\Delta (q(\mathcal S u^*))-\nabla\cdot (b(\mathcal S u^*))\right)\rho=0.
\]
The coefficient $\widetilde b=\nabla( q(\mathcal S u^*))-b(\mathcal S u^*)$ is bounded. Moreover,
\[
\Delta (q(x,\mathcal (S u^*)(x)))=
\Delta_x q(x,(\mathcal S u^*)(x))+2\nabla_x\left( \frac{\partial q}{\partial s}\right) (x,(\mathcal S u^*)(x))\cdot
\nabla(\mathcal S u^*)(x)
\]
\[
+\frac{\partial^2 q}{\partial s ^2}(x,(\mathcal S u^*)(x))|\nabla(\mathcal S u^*)(x)|^2+
\frac{\partial q}{\partial s}(x,(\mathcal S u^*)(x))\Delta(\mathcal S u^*)(x),
\]
and every term is bounded because $q\in C^2_b$, $\nabla(\mathcal S u^*)\in L^\infty$, and $\Delta(\mathcal S u^*)=\mathcal S u^*-u^*\in L^\infty$. The same argument, using $b\in C_b^{1,2}$, gives $\nabla\cdot (b(x,(\mathcal S u^*)(x)))\in L^\infty$. Thus all coefficients entering both equations satisfy the assumptions of Proposition \ref{appendix_parabolic_regularity}.

\section*{Declarations}

\subsection*{\bf Funding} The authors declare that no funds, grants or other support were received during the preparation of this manuscript.

\subsection*{\bf Competing interests} The authors declare that they have no relevant financial or non-financial interests to disclose.

\subsection*{\bf Availability of data and material} Not applicable

\subsection*{\bf Code availability} Not applicable


\begin{thebibliography}{99}

\bibitem{Anita2021} 
\c{S}.-L. Ani\c{t}a, 
Optimal control of stochastic differential equations via Fokker-Planck equations, \it Appl. Math. Optim. \rm 84 (Suppl. 2) (2021), 1555-1583.

\bibitem{Anita2024_1} 
\c{S}.-L. Ani\c{t}a, 
Controlling a nonlinear Fokker-Planck equation via inputs with nonlocal action, \it Appl. Math. Optim. \rm 89 (3) (2024), 68.

\bibitem{anita3} 
\c{S}.-L. Ani\c{t}a, 
Controlling a generalized Fokker-Planck equation via inputs with nonlocal action, \it Nonlin. Anal. \rm 241 (2024), 113476.

\bibitem{AnitaDiPersio2025}
\c{S}.-L. Ani\c{t}a, L. Di Persio,
Optimal control for diffusion processes: existence, maximum principle and financial applications, \it arXiv:2503.05912 [math.OC] \rm (2025).

\bibitem{MarioAlfioJSC}
M. Annunziato, A. Borzì,
A {F}okker-{P}lanck approach to the reconstruction of a cell membrane potential,
\it SIAM J. Sci. Comput. \rm 43 (3) (2021), B623-B649.

\bibitem{MarioAlfioMCA}
M. Annunziato, A. Borz\`{\i},
{F}okker-{P}lanck analysis of superresolution microscopy images,
\it Math. Comput. Appl. \rm 28 (6) (2023), 113.

\bibitem{Barbu1994}
V. Barbu,  
\it Mathematical Methods in Optimization of Differential Systems, \rm Springer, 1994. 

\bibitem{Barbu2020} 
V. Barbu, 
Optimal feedback controllers for a stochastic differential equation with reflection, \it SIAM J. Control Optim. \rm 58 (2) (2020), 986-997.

\bibitem{BarbuRZ2020} 
V. Barbu, M. R\"{o}ckner, D. Zhang, 
Optimal control of nonlinear stochastic differential equations on Hilbert spaces, \it SIAM J. Control Optim. \rm 58 (4) (2020), 2383-2410. 

\bibitem{Brezis2005}
H. Br\'{e}zis,
\it Analyse Fonctionnelle. Th\' {e}orie et Applications, \rm Dunod, Paris, 2005.

\bibitem{Dipersio} 
L. Di Persio, P. Kuchling, 
Optimal control of McKean-Vlasov equations with controlled stochasticity, \it Evol. Eqs. Control Theory \rm 14 (2) (2025), 367-387.

\bibitem{Evans}
L.C. Evans,
\it Partial Differential Equations, 2nd ed., \rm
AMS, 2010.

\bibitem{Figalli2008}
A. Figalli,
Existence and uniqueness of martingale solutions for SDEs with rough or degenerate coefficients,
\it J. Funct. Anal. \rm 254 (2008), 109-153.

\bibitem{Huang2021} 
X. Huang, P. Ren, F.-Y. Wang, 
Distribution dependent stochastic differential equations, \it Front. Math. China \rm 16 (2021), 257-301.


\bibitem{Krylov2008}
N.V. Krylov,
\it Lectures on Elliptic and Parabolic Equations in Sobolev Spaces, \rm Graduate Studies in Mathematics, Vol. 96, American Mathematical Society, 2008.

\bibitem{Ladyzhenskaya1968}
O.A. Ladyzhenskaya, V.A. Solonnikov, N.N. Ural'tseva,
\it Linear and Quasilinear Equations of Parabolic Type, \rm Translations of Mathematical Monographs, Vol. 23, American Mathematical Society, Providence, 1968.

\bibitem{Lieberman1996}
G.M. Lieberman,
\it Second Order Parabolic Differential Equations, \rm World Scientific, 1996.

\bibitem{Lions}
J.-L. Lions,
\it Quelques m\' {e}thodes de r\' {e}solution des probl\`{e}mes aux limites non lin\'{e}aires,
\rm Dunod, 1969.

\bibitem{oksendal} 
B. {{\O{}}}ksendal, 
\it Stochastic Differential Equations. An Introduction with Applications, \rm Fifth Edition, Springer, Berlin, 1998.

\bibitem{RoecknerZhang2010}
M. R\"{o}ckner, X. Zhang,
Weak uniqueness of Fokker–Planck equations with degenerate and bounded coefficients,
\it C.R. Acad. Sci. Paris \rm 348 (2010), 435-438.

\bibitem{Sumin2009}
M.I. Sumin, 
The first variation and Pontryagin's maximum principle in optimal control for partial differential equations,
\it Comput. Math., Math. Phys. \rm 49 (6) (2009), 958-978.

\bibitem{Trevisan2016} 
D. Trevisan, 
Well-posedness of multidimensional diffusion processes with weakly differentiable coefficients, 
\it Electron. J. Probab. \rm 21 (2016), 1-41.

\end{thebibliography}
\end{document}